\documentclass{amsart}
\usepackage{amsmath}
\usepackage{amssymb}
\usepackage[utf8]{inputenc}
\usepackage[english]{babel}
\usepackage[dvipsnames]{xcolor} 
\usepackage{amsthm}
 
\theoremstyle{definition}
\newtheorem{definition}{Definition}[section]
\newtheorem{theorem}[definition]{Theorem}

\newtheorem{lemma}[definition]{Lemma}
\newtheorem{remark}[definition]{Remark}
\newtheorem{example}[definition]{Example}
\newtheorem{proposition}[definition]{Proposition}
\theoremstyle{remark}
\newtheorem*{ack}{Acknowledgments}
\title{Knots not concordant to L-space knots}
\author{Ramazan Yozgyur}
\address{Institute of Mathematics, University of Warsaw, ul. Banacha 2,
02-097 Warsaw, Poland}
\email{ryozgyur@mimuw.edu.pl}
\thanks{The author was supported by  the National Science Center grant 2016/22/E/ST1/00040.}

\begin{document}
\begin{abstract}
In this short note we use methods of Friedl, Livingston and Zentner to show
that there are knots that are not algebraically concordant to a connected sum
of positive and negative L-space knots.
\end{abstract}
\maketitle

\section{Introduction}
A 3-manifold $Y$ is called an L-space if it is a rational homology sphere and its Heegaard-Floer homology has minimal possible rank i.e. $\widehat{HF}(Y)=|H_1(Y,\mathbb{Z})|$.
A knot $K\subset S^3$ is a (positive) L-space knot if a surgery with a sufficiently large positive coefficient is
an L-space. 

L-space knots were introduced by Ozsv\'ath and Szab\'o in \cite{lens} in their attempt to classify
knots such that a surgery on them gives a lens space. In particular, they proved the following
result.

 \begin{theorem}[\expandafter{\cite[Corollary 1.3]{lens}}]\label{thm:lens}
 If $\Delta$ is the Alexander polynomial of an L-space knot, then $\exists$ $a_1,...,a_n$ such that  $a_0<a_1......<a_n$ and
 \begin{equation}\label{eq:lens_poly}
 \Delta=t^{a_0}- t^{a_1}+t^{a_2}....+t^{a_n}
 \end{equation}
 \end{theorem}

The class of L-space knots includes all torus knots. More generally, all algebraic knots are L-space knots; 
see \cite{hed2,GN}. Moreover, in \cite{Hedden} Hedden proves the following result:
\begin{theorem} 
 An L-space knot is strongly quasipositive and fibred.
\end{theorem}
We refer to \cite{Rudolph} for the definition and the properties of strongly quasipositive
knots.

We have the following result of Borodzik and Feller.
\begin{theorem}[see \cite{maciejpeter}]
Every link $L$ is topologically concordant to a strongly quasipositive link $L'$. 
\end{theorem}

In this light, the main theorem of the paper seems a bit surprising.
\begin{theorem}\label{thm:main}
There are knots that are not concordant to any combinations of L-space knots and their mirrors.
\end{theorem}

\section{Proof of Theorem~\ref{thm:main}}
Let us recall the following fact, which can be found e.g. in \cite[Section 8.1]{rahman}.
  \begin{theorem}[Cauchy's bound]
 Suppose $P=\alpha_0+\alpha_1 t+\dots+\alpha_m t^m$ is a complex polynomial 
 with $\alpha_m\neq 0$.
 
 For any root $z$ of $P$ we have $|z|<1+\max_{j<m}\frac{|\alpha_j|}{|\alpha_m|}$.
 \end{theorem}
  From Cauchy's bound we obtain the following proposition.

 \begin{proposition}\label{prop:root_of_alex}
 The Alexander polynomial of an L-space knot has roots with modulus at most 2.
 \end{proposition}

\begin{remark}
From the multiplicativity of the Alexander polynomial under connected sums we infer also that
if $K$ is a connected sum of L-space knots and their mirrors, then
$\Delta_K(t)$ has all roots inside of the disk of radius~2. Note that
by \cite{krcatovich}, a non-trivial connected sum of L-space knots is not an L-space
knot anymore.
\end{remark}

Following \cite{FLZ}, for $n=1,2,\dots$ we define
\begin{equation}\label{eq:defpoly}
P_n=1+nt-(2n+1)t^2+nt^3+t^4.
\end{equation}

We need some properties of roots of this polynomial. 
\begin{theorem}[see \expandafter{\cite[Lemma 4.]{FLZ}}]
%
%
The polynomial $P_n$ is irreducible over $\mathbb{Q}[t]$ and
it has two roots on the unit circle.
\end{theorem} 
Denote these two roots by $\theta_n^+$ and $\theta_n^-$. We will use a
more specific control for one of the other roots of $P_n$.
\begin{lemma}
For $n\ge 1$, the polynomial $P_n$ has a root with modulus greater than $2$.
\end{lemma}
\begin{proof}
We have $P_n(-n-1)=1-2n-3n^2-n^3<0$. On the other hand, $P(-(n+2))=13+10n+2n^2>0$. Therefore
$P_n$ has a real root in the interval $(-n-2,-n-1)$.
\end{proof}

\begin{theorem}\label{thm:cantvanish}
Let $K$ be connected sum of L-space knots and some mirrors of L-space knots, then $\Delta_K $ cannot vanish at $\theta_n^+$. 
\end{theorem}
\begin{proof}
 Since the Alexander polynomial of a connected sum of knots is the product of the Alexander polynomial of each knot, it is enough to prove the theorem for $K$ being an L-space knot.
 
 Suppose $\Delta_K(\theta_n^+)=0$. Then $\gcd(\Delta_K,P_n)$ has positive degree
 and it divides $P_n$. As $P_n$ is irreducible over $\mathbb{Q}[t]$, it follows
 that $P_n|\Delta_K$. But $P_n$ has a root outside a disk of radius $2$ and all
 the roots of $\Delta_K$ are inside this disk.
\end{proof}
\begin{theorem}
Suppose $K$ is a knot that is concordant to a knot $K'$ which is a connected sum of L-space knots and mirrors
of L-space knots. Then the order of the root of $\Delta_K$ at $\theta_n^+$ is an even
number.
\end{theorem}
\begin{proof}
As $K$ and $K'$ are concordant, there exists polynomials $f,g\in\mathbb{Z}[t,t^{-1}]$
such that
\begin{equation}\label{eq:conc}
\Delta_K(t)f(t)f(t^{-1})= \Delta_{K'}(t)g(t)g(t^{-1})
\end{equation}

\emph{Claim.} If $\xi(t)\in\mathbb{Z}[t,t^{-1}]$ vanishes at $\theta_n^+$, then
$\xi^{-1}(t)$ vanishes at $\theta_n^+$ with the same order.

To prove the claim, note that if $\xi$ vanishes at $\theta_n^+$ up to order $s$,
then it is divisible by $(t-\theta_n^+)^s$ and also by $(t-\theta_n^-)^s$ (because
$\xi$ has real coefficients and $\overline{\theta_n^+}=\theta_n^-$).
As $\theta_n^+\theta_n^-=1$ we have
\[(t^{-1}-\theta_n^+)(t^{-1}-\theta_n^{-})=t^{-2}(t-\theta_n^+)(t-\theta_n^-).\]
From the above identity the claim follows readily.

From the claim we conclude that the order of the root of $\Delta_{K'}g(t)g(t^{-1})$
at $\theta_n^+$ is an even number. Using the claim once again, this time for $f(t)$,
we see that \eqref{eq:conc} implies that $\Delta_K(t)$ vanishes at $\theta_n^+$
up to an even power (maybe zero).
\end{proof}
To conclude the proof of Theorem~\ref{thm:main} we will show
that there exist knots such that their Alexander polynomial vanishes at $\theta_n^+$
with an odd order.
As $P_n$ is a symmetric polynomial and $P_n(1)=1$, for any $n$ there exist
a knot $K_n$ such that $\Delta_{K_n}=P_n$, see \cite{Seifert}.
Furthermore, the knot $K_n$ from \cite[Figure 1]{FLZ} has Alexander polynomial $P_n$.
\begin{example}
A notable example of a knot that is not concordant to a combination of L-space
knots is the knot $12n642$. According to KnotInfo web page \cite{info},
its Alexander polynomial is $P_7$. On the other hand, $12n642$ is strongly
quasipositive and fibered. It is also almost positive in the sense of \cite{FLL}.
\end{example}


%
%
\begin{ack}
The author is very grateful to his advisor, Maciej Borodzik, for his help during
preparation of the manuscript. He also expresses his gratitude towards Chuck Livingston and Andras Stipsicz for his comments on the first version of the manuscript.
\end{ack}
\bibliographystyle{amsalpha}
\def\MR#1{}
\bibliography{biblio}
\end{document}